\documentclass{article}
\usepackage{amssymb}
\usepackage{amsfonts}
\usepackage{amsmath}

\newtheorem{theorem}{Theorem}

\newtheorem{definition}[theorem]{Definition}

\newtheorem{remark}[theorem]{Remark}

\newenvironment{proof}[1][Proof]{\noindent\textbf{#1.} }{\ \rule{0.5em}{0.5em}}
\begin{document}

\title{The Stefan problem with mushy region as a scaling limit of stochastic
PDE with turbulent transport}
\author{Ioana CIOTIR\footnote{Normandie University, INSA de Rouen Normandie, LMI (EA 3226 – FR CNRS 3335), 76000 Rouen, France, {\tt\small ioana.ciotir@insa-rouen.fr}}
\and
Franco FLANDOLI\footnote{Scuola normale superiore di Pisa, Piazza dei Cavalieri, 7, 56126 Pisa PI, Italie {\tt\small franco.flandoli@sns.it}}
\and
Dan GOREAC\footnote{\'{E}cole d'Actuariat, Universit\'{e} Laval, QC G1V0B3, Qu\'{e}bec, Canada {\tt\small dan.goreac@act.ulaval.ca} and LAMA, Univ Gustave Eiffel, UPEM, Univ Paris Est Creteil, CNRS, F-77447 Marne-la-Vall\'{e}e, France, {\tt\small dan.goreac@univ-eiffel.fr}}
}
\maketitle

\begin{abstract}
This work establishes a scaling limit theorem for the Stefan problem incorporating a mushy region, demonstrating that solutions to stochastic variants with turbulent transport terms converge to the solution to a deterministic partial differential equation. The analysis builds upon recent advances in stochastic phase-change modeling and turbulent flow mathematics in \cite{IFD1}. In the physical interpretation of an ice melting process, our result shows that turbulence accelerates ice melting.
\end{abstract}



\section{Introduction}

\subsection{The physical model}

In the present work we are interested in the scaling limit of a Stefan type
problem describing the melting/solidification process with mushy region and
turbulent transport noise. We know by classical theory that the model can be
written as%
\begin{equation}
\left\{ 
\begin{array}{l}
\frac{d}{dt}\gamma \left( \theta \right) -\textit{div}\left( k\left( \theta
\right) \nabla \theta \right) +u\cdot \nabla \eta \left( \theta \right) =F
\\ 
\theta \left( 0,\xi \right) =\theta _{0}\left( \xi \right) ,%
\end{array}%
\right.  \label{equ1}
\end{equation}%
where $\gamma \left( r\right) =C\left( r\right) +l\cdot H\left( r\right) $
with 
\begin{equation*}
C\left( r\right) =\left\{ 
\begin{array}{cc}
C_{1}r, & \quad r\leq 0 \\ 
C_{2}r, & \quad r>0,%
\end{array}%
\right.
\end{equation*}%
where $H$ is the multivalued Heaviside function%
\begin{equation*}
H\left( r\right) =\left\{ 
\begin{array}{ll}
0, & \quad r<0 \\ 
\left[ 0,1\right] , & \quad r=0 \\ 
1, & \quad r>0,%
\end{array}%
\right.
\end{equation*}
and%
\begin{equation*}
k\left( r\right) =\left\{ 
\begin{array}{cc}
k_{1}, & \quad r\leq 0 \\ 
k_{2,} & \quad r>0.%
\end{array}%
\right.
\end{equation*}

The solution $\theta $ to the previous equation is interpreted as the
temperature which is assumed to be positive in the liquid phase and negative
in the solid phase. We denoted by $C_{1}$ and $C_{2}$ the specific heat of
the two phases and by $k_{1}$ and $k_{2}$ the thermal conductivity. The
constant $l$ is the latent heat intervening in the melting/solidification
phenomena.

The function $\eta $ is assumed to be Lipschitz and smooth, vanishing in the
solid phase and in a small region of the liquid phase which is situated very
close to the boundary between ice and water. From the physical point of view
it is coherent to assume this because the term $u\cdot \nabla \eta \left(
\theta \right) $ is modeling the turbulence which can't be present in the
solid phase nor on the boundary of the liquid phase.

The function $F$ is to be fixes and assumed to be the heating source which
is smooth enough from the mathematical point of view.

Concerning the velocity of the fluid, a good model for the turbulent fluid
has the form%
\begin{equation*}
u\left( t,\xi \right) =\sum_{k=1}^{\infty }\alpha _{k}\sigma _{k}\circ
d\beta _{k}
\end{equation*}%
where we denoted by $\left\{ \alpha _{k}\right\} _{k}$ a sequence of
positive constants which were chosen such that the assumptions below to be
satisfied. The sequence $\left\{ \sigma _{k}\right\} _{k}$ is formed by
divergence free smooth vector fields, $\left\{ \beta _{k}\right\} $ is a
sequence of independent Brownian motions and the integral is considered in
the sense of Stratonovich, thus aligning with the Wong-Zakai principle that establishes convergence of approximations to stochastic differential equations (see rigorous analysis in \cite{Debussche}). This approach is motivated by the emergence of turbulent effects in the liquid phase during phase transitions, where thermal gradients between phases induce complex fluid instabilities.

The noise model described has gained significant attention in contemporary research, particularly regarding its mathematical properties and physical implications. For foundational aspects and recent developments, we direct readers to key works including \cite{Flandoli2}, \cite{Flandoli1}, \cite{Flandoli5}, \cite{Flandoli-book}, and \cite{Flandoli-Pappalettera}. These investigations establish crucial connections between stochastic calculus frameworks and observed turbulent phenomena in multi-phase systems.\\

In order to study this equation we need first to note that%
\begin{equation*}
-\textit{div}\left( k\left( \theta \right) \nabla \theta \right) =-\Delta
K\left( \theta \right)
\end{equation*}%
where $K$ is%
\begin{equation*}
K\left( r\right) =\left\{ 
\begin{array}{cc}
k_{1}r, & \quad r\leq 0 \\ 
k_{2}r, & \quad r>0,%
\end{array}%
\right.
\end{equation*}%
and we rewrite the equation (\ref{equ1}) as%
\begin{equation*}
\left\{ 
\begin{array}{l}
d\gamma \left( \theta \right) -\Delta K\left( \theta \right) dt+u\cdot
\nabla \eta \left( \theta \right) =F \\ 
\theta \left( 0,\xi \right) =\theta _{0}\left( \xi \right) .%
\end{array}%
\right.
\end{equation*}

We shall now take the change of variable $\gamma \left( \theta \right) =X$
and we get%
\begin{equation*}
\left\{ 
\begin{array}{l}
dX-\Delta K\left( \gamma ^{-1}\left( X\right) \right) dt+u\cdot \nabla \eta
\left( \gamma ^{-1}\left( X\right) \right) =F \\ 
X\left( 0,\xi \right) =x_{0}\left( \xi \right) .%
\end{array}%
\right.
\end{equation*}

For more clearness we shall denote by $\Psi _{0}=K\circ \gamma ^{-1}$ and $%
\Gamma _{0}=\eta \circ \gamma ^{-1}$ and rewrite the equation as follows%
\begin{equation*}
\left\{ 
\begin{array}{l}
dX-\Delta \Psi _{0}\left( X\right) dt+u\cdot \nabla \Gamma _{0}\left(
X\right) =F \\ 
X\left( 0,\xi \right) =x_{0}\left( \xi \right) .%
\end{array}%
\right.
\end{equation*}

Keeping in mind the physical application and the fact that the interface
between ice and water is not sharp, it is reasonable to assume that the
mushy region between the two phases is modeled by replacing the Heaviside
from $\gamma $ by a smoothed function and consequently to replace $\gamma $
by the corresponding smoothed $\widetilde{\gamma }$ and also $\Psi _{0}$ and 
$\Gamma _{0}$ by the smooth version denoted by $\Psi $ and $\Gamma $.

\bigskip

\textbf{Hypotheses}

In this setting we assume that

\begin{itemize}
\item $\Psi $ and $\Gamma $ are assumed to be strictly monotone such that $%
\Psi ^{\prime }\geq \psi _{0}>0$ and $\Gamma ^{\prime }\geq \gamma _{0}>0$.
Note that $\Psi $ and $\Gamma $ as assumed before are Lipschitz continuos,
and to be null in zero.

\item $\widetilde{\gamma }^{-1}$ is assumed to satisfy that $\left( 
\widetilde{\gamma }^{-1}\right) ^{\prime }<1$. (This assumption is not
restrictive and is necessary in order to apply the existence result for the
stochastic equation.)
\end{itemize}

\bigskip

We shall now treat rigorously the following equation%
\begin{equation}
\left\{ 
\begin{array}{l}
dX-\Delta \Psi \left( X\right) dt+u\cdot \nabla \Gamma \left( X\right) =F \\ 
X\left( 0,\xi \right) =x_{0}\left( \xi \right) .%
\end{array}%
\right.  \label{equ2}
\end{equation}

For this reason we shall introduce now the suitable mathematical framework.

\subsection{Functional setting}

\bigskip

We consider $\Pi ^{2}=\mathbb{R}^{2}/%
\mathbb{Z}
^{2}$ the two dimensional torus and $%
\mathbb{Z}
_{0}^{2}=%
\mathbb{Z}
^{2}\backslash \left\{ 0\right\} $ the nonzero lattice points. Let $\left(
H^{s,p}\left( \Pi ^{2}\right) ,\left\Vert \cdot \right\Vert
_{H^{s,p}}\right) ,$ $s\in \mathbb{R}$, $p\in \left( 1,\infty \right) $ be a
Bessel space of zero mean periodic functions.

We denote $H^{s}\left( \Pi ^{2}\right) =H^{s,2}\left( \Pi ^{2}\right) $ and
we denote by $\left\langle \cdot ,\cdot \right\rangle _{H^{s}}$ the
corresponding scalar product. For $s>0$ we can consider $H^{-s}$ the dual of 
$H^{s}$ and denote $\left\langle \cdot ,\cdot \right\rangle _{H^{-s},H^{s}}$
the duality pairing.

The Bessel space of zero mean vector field are defined as 
\begin{equation*}
\mathbf{H}^{s,p}=\left\{ \left( u_{1},u_{2}\right) ^{t}\left\vert
~u_{1},u_{2}\in H^{s,p}\left( \Pi ^{2}\right) \right. \right\}
\end{equation*}%
endowed with the scalar product%
\begin{equation*}
\left\langle u,v\right\rangle _{\mathbf{H}^{s}}=\left\langle
u_{1},v_{1}\right\rangle _{H^{s}}+\left\langle u_{2},v_{2}\right\rangle
_{H^{s}},\quad \forall u,v\in \mathbf{H}\text{ and }\forall s\in \mathbb{R}%
\text{.}
\end{equation*}%
Again, for $s=0$ we denote $\mathbf{H}^{0,2}=\mathbf{L}^{2}.$

If we consider $Z$ a separable Hilbert space endowed with the norm $%
\left\Vert \cdot \right\Vert _{Z},$ we denoted by $C_{\mathcal{F}}^{W}\left( %
\left[ 0,T\right] ;Z\right) $ the space of weakly continuous processes $%
\left( X_{t}\right) _{t\in \left[ 0,T\right] }$ defined on a filtered
probability space $\left( \Omega ,\mathcal{F},\left( \mathcal{F}_{t}\right)
_{t\geq 0},\mathbb{P}\right) ,$ with values in $Z$ and adopted, such that%
\begin{equation*}
\mathbb{E}\left[ \underset{t\in \left[ 0,T\right] }{\sup }\left\Vert
X_{t}\right\Vert _{Z}^{2}\right] <\infty
\end{equation*}%
and by $L_{\mathcal{F}}^{p}\left( 0,T;Z\right) ,$ $p\in \left[ 1,\infty
\right) $ the space of progressively measurable processes $\left(
X_{t}\right) _{t\in \left[ 0,T\right] }$ such that 
\begin{equation*}
\mathbb{E}\left[ \int_{0}^{T}\left\Vert X_{t}\right\Vert _{Z}^{2}\right]
<\infty .
\end{equation*}

\bigskip

\subsection{The mathematical model}

\bigskip

We can now rigorously define as follows, the mathematical model which allows
us to study equation (\ref{equ2}).

First, in the turbulence term $u\cdot \nabla \Gamma \left( X\right) $ we
take the Stratonovich interpretation%
\begin{equation*}
u\left( t,\xi \right) =\sum_{k\in 
\mathbb{Z}
_{0}^{2}}\alpha _{k}\sigma _{k}\circ d\beta _{k}
\end{equation*}%
and we get 
\begin{equation}
u\cdot \nabla \Gamma \left( X\right) =\sum_{k\in 
\mathbb{Z}
_{0}^{2}}\alpha _{k}\sigma _{k}\nabla \Gamma \left( X\right) \circ d\beta
_{k}  \label{noise}
\end{equation}%
where $\Gamma $ is assumed to be Lipschitz and smooth and such that $\Gamma
\left( r\right) =0$ on $\left( -\infty ,\varepsilon \right) $ for some $%
\varepsilon >0$ very mall and $\left\vert \Gamma \left( r\right) ^{\prime
}\right\vert \leq const.$

Moreover $\left( \alpha _{k}\right) _{k}\in l^{2}\left( 
\mathbb{Z}
^{2}\right) $ satisfies 
\begin{equation}
\sum\limits_{k\in 
\mathbb{Z}
_{0}^{2}}\frac{\alpha _{k}^{2}}{\vert k\vert^2}=1,\text{ and }\alpha _{k}=\alpha _{l}\text{ if }%
\left\vert k\right\vert =\left\vert l\right\vert .  \label{ip1}
\end{equation}

Finally $\left( \sigma _{k}\right) _{k}$ is a standard orthonormal basis of
divergence free vectors fields in $\mathbf{L}^{2}.$ More precisely, for a
two dimensional torus $\Pi ^{2}$ which will be seen as $\left[ -1/2,1/2%
\right] $ endowed with periodic boundary conditions, we have the usual
trigonometric basis of $L^{2}\left( \Pi ^{2}\right) $ which is%
\begin{equation*}
e_{k}\left( x\right) =\sqrt{2}\left\{ 
\begin{array}{cc}
\cos \left( 2\pi kx\right) , & k\in 
\mathbb{Z}
_{+}^{2} \\ 
\sin \left( 2\pi kx\right) , & k\in 
\mathbb{Z}
_{-}^{2}%
\end{array}%
\right. \text{ for all }x\in \Pi ^{2}
\end{equation*}%
for 
\begin{equation*}
\mathbb{Z}
_{+}^{2}=\left\{ k\in 
\mathbb{Z}
_{0}^{2}~\left\vert \left( k_{1}>0\right) \text{ or }\left(
k_{1}=0,k_{2}>0\right) \right. \right\}
\end{equation*}%
and $%
\mathbb{Z}
_{-}^{2}=-%
\mathbb{Z}
_{+}^{2}$.

We define 
\begin{equation*}
\sigma _{k}\left( x\right) =\frac{k^{\bot }}{\left\vert
k\right\vert ^{2}}e_{k}\left( x\right) ,\quad k\in 
\mathbb{Z}
_{0}^{2}\text{ with }k^{\bot }=\left( k_{2},-k_{1}\right) .
\end{equation*}

Finally $\left( \beta _{k}\right) _{k}$ is a sequence of independent standard, one-dimensional
Brownian motions supported on the probability space introduced before.

The Stratonovich noise that we introduced previously can be formulated in It%
\^{o} form by the following transformation%
\begin{eqnarray}
\sum_{k\in 
\mathbb{Z}
_{0}^{2}}\alpha _{k}\sigma _{k}\nabla \Gamma \left( X\right) \circ d\beta
_{k} &=&\sum_{k\in 
\mathbb{Z}
_{0}^{2}}\alpha _{k}\sigma _{k}\nabla \Gamma \left( X\right) d\beta _{k}
\label{I-S} \\
&&-\frac{1}{2}\sum_{k\in 
\mathbb{Z}
_{0}^{2}}\alpha _{k}^{2}\textit{div}\left[ \left( \Gamma ^{\prime }\left(
X\right) \right) ^{2}\sigma _{k}\otimes \sigma _{k}\nabla X\right] dt. 
\notag
\end{eqnarray}

By applying Lemma 2.6 from \cite{Flandoli-Luo2020} we have that%
\begin{equation*}
\sum_{k\in 
\mathbb{Z}
_{0}^{2}}\alpha _{k}^{2}\left( \sigma _{k}\otimes \sigma _{k}\right) =\frac{1%
}{2}I_{2},
\end{equation*}%
where $I_{2}$ is a two-dimensional identity matrix.

We denote by 
\begin{equation*}
g\left( r\right) =\frac{1}{4}\int_{0}^{r}\left( \Gamma ^{\prime }\left(
x\right) \right) ^{2}dx,\quad r\in \mathbb{R}
\end{equation*}%
which satisfies $g\left( 0\right) =0$ and is Lipschitz from the properties
of $\Gamma .$

Going back to (\ref{I-S}) we get that%
\begin{eqnarray}
&&\sum_{k\in 
\mathbb{Z}
_{0}^{2}}\alpha _{k}\sigma _{k}\nabla \Gamma \left( X\right) \circ d\beta
_{k}  \label{I-S2} \\
&=&\sum_{k\in 
\mathbb{Z}
_{0}^{2}}\alpha _{k}\sigma _{k}\nabla \Gamma \left( X\right) d\beta _{k}-%
\frac{1}{4}\textit{div}\left[ \left( \Gamma ^{\prime }\left( X\right) \right)
^{2}\nabla X\right] dt  \notag \\
&=&\sum_{k\in 
\mathbb{Z}
_{0}^{2}}\alpha _{k}\sigma _{k}\nabla \Gamma \left( X\right) d\beta
_{k}-\Delta g\left( X\right) .  \notag
\end{eqnarray}

We can rigorously write equation (\ref{equ2}) as%
\begin{equation}
\left\{ 
\begin{array}{l}
dX-\Delta \Psi \left( X\right) dt-\Delta g\left( X\right) dt+\sum_{k\in 
\mathbb{Z}
_{0}^{2}}\alpha _{k}\sigma _{k}\nabla \Gamma \left( X\right) d\beta _{k}=F
\\ 
X\left( 0,\xi \right) =x_{0}%
\end{array}%
\right.   \label{ecuf}
\end{equation}

The existence of solution for the equation above was recently studied in the
case of a general domain in \cite{IFD1} in the sense of the definition
below. This result obviously applies in the present case by choosing a
particular domain and a particular form of the orthonormal basis which
appears in the construction of the noise.

\begin{definition}
Let $x\in L^{2}\left( \Pi ^{2}\right) .$ We say that equation (\ref{ecuf})
has a weak solution if there exist

\begin{itemize}
\item a filtered reference probability space $\left( \Omega ,\mathcal{F}%
,\left( \mathcal{F}_{t}\right) _{t\geq 0},\mathbb{P}\right) $,

\item a sequence of independent $\mathcal{F}_{t}$ Brownian motions,

\item an $H^{-1}\left( \Pi ^{2}\right) -$valued continuous $\mathcal{F}_{t}-$%
adapted process $X\in L^{2}\left( 0,T;L^{2}\left( \Pi ^{2}\right) \right) $
\end{itemize}

and the following holds true%
\begin{eqnarray*}
\left( X\left( t\right) ,e_{j}\right) _{2} &=&\left( x,e_{j}\right)
_{2}+\int_{0}^{t}\left( F\left( s\right) ,e_{j}\right) _{2}ds \\
&&+\int_{0}^{t}\left( \Psi \left( X\left( s\right) \right) ,\Delta
e_{j}\right) _{2}ds+\int_{0}^{t}\left( g\left( X\left( s\right) \right)
,\Delta e_{j}\right) _{2}ds \\
&&+\sum_{k\in 
\mathbb{Z}
_{0}^{2}}\int_{0}^{t}\alpha _{k}\left( \sigma _{k}\Gamma \left( X\left(
s\right) \right) ,\nabla e_{j}\right) d\beta _{k}\left( s\right) .
\end{eqnarray*}
\end{definition}

\bigskip

\section{The scaling limit and the main result}

\bigskip

Following the idea introduced in \cite{galeati} we consider a sequence $%
\left\{ \alpha ^{N}\right\} _{N\in \mathbb{N}}\subseteq l^{2}\left( \mathbb{Z%
}_{0}^{2}\right) $ constructed as before and satisfying%
\begin{equation}
\underset{N\rightarrow \infty }{\lim }\left\Vert \alpha ^{N}\right\Vert
_{l^{\infty }}=0 , \label{ip2}
\end{equation}%
and we denote $X^{N}$ the corresponding solutions of equation (\ref{ecuf})
with $\left\{ \alpha _{k}^{N}\right\} $ instead of $\left\{ \alpha
_{k}\right\} $ in the construction of the noise.

In the present work we prove that the law of $X^{N}$ converges in the usual weak (or, more precisely weak-*) sense to a
Dirac mass concentrated on the unique solution to the following deterministic porous media
equation%
\begin{equation}
\left\{ 
\begin{array}{l}
dX-\Delta \Psi \left( X\right) dt-\Delta g\left( X\right) dt=F \\ 
X\left( 0,\xi \right) =x.%
\end{array}%
\right.  \label{ecul}
\end{equation}
The new term $\Delta g(X)$ enhances the diffusion given byt the term $\Delta \Psi(X)$. Therefore our result may be interpreted as ice melting enhancement by a turbulent fluid.\\
Since $\Psi $ and $g$ are both maximal monotone and Lipschitz operators (as real-valued functions, hence as $L^2$-defined functionals), the
equation \eqref{ecul} has a unique strong solution owing to the classical PDE theory
(see e.g. \cite{nonlin}).

We can write the solution to equation (\ref{ecul}) as a PDE weak one
\[X\in L^{2}\left( \left( 0,T\right) \times \Pi ^{2}\right) \cap C\left( %
\left[ 0,T\right] \times H^{-1}\left( \Pi ^{2}\right) \right)\] in the
following form%
\begin{eqnarray*}
\left( X\left( t\right) ,e_{j}\right) _{2} &=&\left( x,e_{j}\right)
_{2}+\int_{0}^{t}\left( F\left( s\right) ,e_{j}\right) _{2}ds \\
&&+\int_{0}^{t}\left( \Psi \left( X\left( s\right) \right) ,\Delta
e_{j}\right) _{2}ds+\int_{0}^{t}\left( g\left( X\left( s\right) \right)
,\Delta e_{j}\right) _{2}ds,
\end{eqnarray*}%
for all test functions $e_{j}$ provided above.

We are now able to state and prove the main result of the paper.

\begin{theorem}
Let $\left\{ \alpha ^{N}\right\} _{N\in \mathbb{N}}\subseteq l^{2}\left( 
\mathbb{Z}_{0}^{2}\right) $ be a sequence satisfying the assumptions (\ref%
{ip1}) and (\ref{ip2}). We denote \begin{itemize}
    \item by $X^{N}$ the corresponding solutions
to the equations (\ref{ecuf}), where $\alpha$ is replaced with $(\alpha^N)$, and 
\item by $\nu ^{N}:=\mathbb{P}_{X^N}$ the law of these solutions supported by $%
L^{2}\left( \left( 0,T\right) \times \Pi ^{2}\right) \cap C\left( \left[ 0,T%
\right] ;H^{-1}\left( \Pi ^{2}\right) \right)$.
\end{itemize}
Then, the family $\left\{
\nu ^{N}\right\} $ is tight on $C\left( \left[ 0,T\right] ;H^{-1}\left( \Pi
^{2}\right) \right) $, and it converges weakly to the Dirac measure $\delta _{X}$%
\ where $X$ is the unique solution of the equation (\ref{ecul}).
\end{theorem}

\begin{proof}
Before providing the main steps of the proof, let us emphasize that, throughout the argument, we drop the $F$ term, as it does not constitute a specific technical difficulty. We believe that the current version is easier to read and grasp the main specificity.\\
\textbf{Step I (The L}$^{2}$\textbf{\ estimate)}

We begin with employing It\^{o}'s formula applied to the squared norm in $L^{2}\left( \Pi ^{2}\right) $ of the stochastic process given by (\ref%
{ecuf}) and corresponding to the $\alpha ^{N}$\ coefficients to get%
\begin{eqnarray}
&&\left\Vert X^{N}\left( t\right) \right\Vert _{2}^{2}+2\int_{0}^{t}\left(
\nabla \Psi \left( X^{N}\left( s\right) \right) ,\nabla X^{N}\left( s\right)
\right) _{2}ds  \label{estim1} \\
&&+2\int_{0}^{t}\left( \nabla g\left( X^{N}\left( s\right) \right) ,\nabla
X^{N}\left( s\right) \right) _{2}ds  \notag \\
&=&\left\Vert x\right\Vert _{2}^{2}+\frac{1}{2}\int_{0}^{t}\left\Vert \nabla
\Gamma \left( X^{N}\left( s\right) \right) \right\Vert _{2}^{2}ds  \notag \\
&&+2\sum_{k\in 
\mathbb{Z}
_{0}^{2}}\int_{0}^{T}\left( \alpha _{k}^{N}\sigma _{k}\Gamma \left(
X^{N}\left( s\right) \right) ,\nabla X^{N}\right) _{2}d\beta _{k}\left(
s\right) .  \notag
\end{eqnarray}

We consider, separately%
\begin{eqnarray*}
2\int_{0}^{t}\left( \nabla g\left( X^{N}\left( s\right) \right) ,\nabla
X^{N}\left( s\right) \right) _{2}ds &=&\frac{2}{4}\int_{0}^{t}\left( \left(
\Gamma ^{\prime }\right) ^{2}\left( X^{N}\left( s\right) \right) ,\left\vert
\nabla X^{N}\left( s\right) \right\vert ^{2}\right) _{2}ds \\
&=&\frac{1}{2}\int_{0}^{t}\left\Vert \nabla \Gamma \left( X^{N}\left(
s\right) \right) \right\Vert _{2}^{2}ds
\end{eqnarray*}%
and 
\begin{eqnarray*}
\left( \alpha _{k}^{N}\sigma _{k}\Gamma \left( X^{N}\left( s\right) \right)
,\nabla X^{N}\right) _{2} &=&\left( \alpha _{k}^{N}\nabla \widetilde{\Gamma }%
\left( X^{N}\left( s\right) \right) ,\sigma _{k}\right) _{2} \\
&=&-\left( \alpha _{k}^{N}\widetilde{\Gamma }\left( X^{N}\left( s\right)
\right) ,\textit{div}\ \sigma _{k}\right) _{2}=0,
\end{eqnarray*}%
where $\widetilde{\Gamma }$\ is a primitive of $\Gamma $, and the last equality follows from the divergence condition we have mentioned when introducing the $\sigma_k$ family.

Going back to (\ref{estim1}), we get, $\mathbb{P}$-a.s. 
\begin{equation*}
\left\Vert X^{N}\left( t\right) \right\Vert _{2}^{2}+2\int_{0}^{t}\left(
\nabla \Psi \left( X^{N}\left( s\right) \right) ,\nabla X^{N}\left( s\right)
\right) _{2}ds\leq \left\Vert x\right\Vert _{2}^{2},
\end{equation*}%
for all $t\geq 0$, and, due to the strong monotonicity of $\Psi $,\ we obtain that 
\begin{equation}
\underset{t\in \left[ 0,T\right] }{\sup }\left\Vert X^{N}\left( t\right)
\right\Vert _{2}^{2}+2\int_{0}^{t}\psi _{0}\left\vert X^{N}\left( s\right)
\right\vert _{H^{1}\left( \Pi ^{2}\right) }^{2}ds\leq \left\Vert
x\right\Vert _{2}^{2}.  \label{estim3}
\end{equation}

We obtain that, along some subsequence, still indexed by $n$ to simplify reading,%
\begin{eqnarray*}
X^{N} &\rightharpoonup &X\text{ weakly}^{\ast }\text{ in }\ L^2(\Omega;L^{\infty }\left(
0,T;L^{2}\left( \Pi ^{2}\right) \right)), \\
X^{N} &\rightharpoonup &X\text{ weakly in }L^2(\Omega;L^{2}\left( 0,T;H^{1}\left( \Pi
^{2}\right) \right)).
\end{eqnarray*}

Keeping in mind that $\Psi $\ and $g$ are Lipschitz continuous we also have
that 
\begin{eqnarray*}
\Psi \left( X^{N}\right) &\rightharpoonup &\chi \text{ weakly in }%
L^2(\Omega;L^{2}\left( 0,T;H^{1}\left( \Pi ^{2}\right) \right)), \\
g\left( X^{N}\right) &\rightharpoonup &\rho \text{ weakly in }L^2(\Omega;L^{2}\left( 
0,T;H^{1}\left( \Pi ^{2}\right) \right) ).
\end{eqnarray*}

In order to identify the limits $\Psi \left( X\right) =\chi $ and $g\left(
X\right) =\rho $ we need some strong convergence that will be observed as
explained in the next step.

\bigskip 

\textbf{Step II (Tightness)}

For each $r\geq 2$ there is a constant $C$ independent of $N$ such that for
any $t$ and $s$ such that $0\leq s\leq t\leq T$ we take the difference%
\begin{eqnarray}
\left( X^{N}\left( t\right) -X^{N}\left( s\right) ,e_{j}\right) _{2}
&=&\left( x,e_{j}\right) _{2}+\int_{s}^{t}\left( F\left( l\right)
,e_{j}\right) _{2}dl  \label{estim2} \\
&&+\int_{s}^{t}\left( \Psi \left( X^{N}\left( l\right) \right) ,\Delta
e_{j}\right) _{2}dl+\int_{s}^{t}\left( g\left( X^{N}\left( l\right) \right)
,\Delta e_{j}\right) _{2}dl  \notag \\
&&+\sum_{k\in 
\mathbb{Z}
_{0}^{2}}\int_{s}^{t}\alpha _{k}\left( \sigma _{k}\Gamma \left( X^{N}\left(
l\right) \right) ,\nabla e_{j}\right) d\beta _{k}\left( l\right)   \notag \\
&=&I_{s,t}^{1}+I_{s,t}^{2}+I_{s,t}^{3}+I_{s,t}^{4},  \notag
\end{eqnarray}%
and we compute $\mathbb{E}\left( \left( X^{N}\left( t\right) -X^{N}\left(
s\right) ,e_{j}\right) _{2}^{r}\right) $ as follows%
\begin{equation*}
\mathbb{E}\left[ \left\vert I_{s,t}^{1}\right\vert ^{r}\right] =\mathbb{E}%
\left[ \left\vert \int_{s}^{t}\left( F\left( l\right) ,e_{j}\right)
_{2}dl\right\vert ^{r}\right] \leq \left( t-s\right) ^{r},
\end{equation*}%
\begin{eqnarray*}
\mathbb{E}\left[ \left\vert I_{s,t}^{2}\right\vert ^{r}\right]  &=&\mathbb{E}%
\left[ \left\vert \int_{s}^{t}\left( \Psi \left( X^{N}\left( l\right)
\right) ,\Delta e_{j}\right) _{2}dl\right\vert ^{r}\right]  \\
&\leq &\mathbb{E}\left[ \left\vert \int_{s}^{t}\left\Vert \Psi \left(
X^{N}\left( l\right) \right) \right\Vert _{2}\left\Vert \Delta
e_{j}\right\Vert _{\infty }dl\right\vert ^{r}\right]  \\
&\leq &C\lambda _{j}^{r}\left( t-s\right) ^{r},
\end{eqnarray*}%
\begin{eqnarray*}
\mathbb{E}\left[ \left\vert I_{s,t}^{3}\right\vert ^{r}\right]  &=&\mathbb{E}%
\left[ \left\vert \int_{s}^{t}\left( g\left( X^{N}\left( l\right) \right)
,\Delta e_{j}\right) _{2}dl\right\vert ^{r}\right]  \\
&\leq &\mathbb{E}\left[ \left\vert \int_{s}^{t}\left\Vert g\left(
X^{N}\left( l\right) \right) \right\Vert _{2}\left\Vert \Delta
e_{j}\right\Vert _{\infty }dl\right\vert ^{r}\right]  \\
&\leq &C\lambda _{j}^{r}\left( t-s\right) ^{r},
\end{eqnarray*}%
(using the fact that $g$ is Lipschitz because $g^{\prime }$ is bounded as $%
\Gamma ^{\prime }$)%
\begin{eqnarray*}
\mathbb{E}\left[ \left\vert I_{s,t}^{4}\right\vert ^{r}\right]  &=&\mathbb{E}%
\left[ \left\vert \sum_{k\in 
\mathbb{Z}
_{0}^{2}}\int_{s}^{t}\alpha _{k}\left( \sigma _{k}\Gamma \left( X^{N}\left(
l\right) \right) ,\nabla e_{j}\right) d\beta _{k}\left( l\right) \right\vert
^{r}\right]  \\
&\leq &C\mathbb{E}\left[ \left\vert \int_{s}^{t}\left( \Gamma \left(
X^{N}\left( l\right) \right) ,\sum_{k\in 
\mathbb{Z}
_{0}^{2}}\alpha _{k}\sigma _{k}\nabla e_{j}\right) _{2}^{2}dl\right\vert ^{%
\frac{r}{2}}\right]  \\
&\leq &C\lambda _{j}^{\frac{r}{2}}\left\vert t-s\right\vert ^{\frac{r}{2}}.
\end{eqnarray*}

Going back in (\ref{estim2}) we get that%
\begin{equation*}
\mathbb{E}\left( \left( X^{N}\left( t\right) -X^{N}\left( s\right)
,e_{j}\right) _{2}^{r}\right) \leq C\left\vert t-s\right\vert ^{\frac{r}{2}%
}\left( \lambda _{j}^{\frac{r}{2}}+\lambda _{j}^{r}\right) .
\end{equation*}

Using this estimate we can show that for some $\beta >4$ and $r\geq 4$ we
have 
\begin{equation*}
\mathbb{E}\left[ \left\Vert X^{N}\left( t\right) -X^{N}\left( s\right)
\right\Vert _{H^{-\beta }}^{r}\right] \leq C\left\vert t-s\right\vert ^{%
\frac{r}{2}}
\end{equation*}%
for some $C$ independent of $N$.

For details, see \ Lemma 5 from \cite{IFD1}. Then, for $\alpha \in \left( 
\frac{1}{r},\frac{1}{2}\right) $ we have 
\begin{equation*}
\mathbb{E}\left[ \int_{0}^{T}\int_{0}^{T}\frac{\left\Vert X^{N}\left(
t\right) -X^{N}\left( s\right) \right\Vert _{H^{-\beta }}^{r}}{\left\vert
t-s\right\vert ^{1+\alpha r}}dtds\right] \leq C.
\end{equation*}

Combining the relation above with the estimate (\ref{estim3}) from $Step~I$
we get that 
\begin{equation}
\mathbb{E}\underset{t\in \left[ 0,T\right] }{\sup }\left\Vert X^{N}\left(
t\right) \right\Vert _{2}^{2}+\mathbb{E}\int_{0}^{t}\int_{0}^{t}\frac{%
\left\Vert X^{N}\left( t\right) -X^{N}\left( s\right) \right\Vert
_{H^{-\beta }}^{r}}{\left\vert t-s\right\vert ^{1+\alpha r}}dtds\leq C.
\label{estim4}
\end{equation}

We can use the compactness result in Corollary 5 from \cite{simon} and
obtain that 
\begin{equation*}
L^{\infty }\left( 0,T;L^{2}\left( \Pi ^{2}\right) \right) \cap W^{\alpha
,r}\left( 0,T;H^{-\beta }\left( \Pi ^{2}\right) \right) \subset C\left( 
\left[ 0,T\right] ;H^{-1}\left( \Pi ^{2}\right) \right) ,
\end{equation*}%
with compact inclusion.

More precisely, if we take 
\begin{equation*}
K_{R}=\left\{ f\in C\left( \left[ 0,T\right] ;H^{-1}\left( \Pi ^{2}\right)
~\left\vert \left\Vert f\right\Vert _{L^{\infty }\left( 0,T;L^{2}\left( \Pi
^{2}\right) \right) }+\left\Vert f\right\Vert _{W^{\alpha ,r}\left(
0,T;H^{-\beta }\left( \Pi ^{2}\right) \right) }\leq R\right. \right)
\right\} ,
\end{equation*}%
we have that $K_{R}$ is compact in $C\left( \left[ 0,T\right] ;H^{-1}\left(
\Pi ^{2}\right) \right) $.

Then, for all $\varepsilon >0$ we have by using Markov's inequality that 
\begin{equation*}
\nu ^{N}\left( K_{R}^{c}\right) \leq \frac{C}{R}\leq \varepsilon 
\end{equation*}%
for $R$ sufficiently large, and we can conclude that the family of laws $%
\left\{ \nu ^{N}\right\} _{N}$ is tight in the space $C\left( \left[ 0,T%
\right] ;H^{-1}\left( \Pi ^{2}\right) \right) $.

\bigskip 

\textbf{Step III (Passage to the limit)}

By using the Skorokhod representation theorem we can find an auxiliary
probability space $\left( \widetilde{\Omega },\widetilde{\mathcal{F}},%
\widetilde{\mathbb{P}}\right) $ and the processes $\left( \widetilde{X}^{N},%
\widetilde{W}^{N}=\left\{ \widetilde{\beta }^{N,k}\right\} _{k\in \mathbb{Z}%
_{0}^{2}}\right) $ and $\left( \widetilde{X},\widetilde{W}=\left\{ 
\widetilde{\beta }^{k}\right\} _{k\in \mathbb{Z}_{0}^{2}}\right) $ such that%
\begin{eqnarray*}
\widetilde{X}^{N} &\rightarrow &\widetilde{X}\text{ in }C\left( \left[ 0,T%
\right] ;H^{-1}\left( \Pi ^{2}\right) \right) \text{ }\mathbb{P}\text{-a.s.}
\\
\widetilde{\beta }^{N,k} &\rightarrow &\widetilde{\beta }^{k}\text{ in }%
C\left( \left[ 0,T\right] ;\mathbb{Z}_{0}^{2}\right) \text{ }\mathbb{P}\text{%
-a.s.}
\end{eqnarray*}%
where the convergence of $\widetilde{W}^{N}$ to $\widetilde{W}$ can be seen
as the uniform convergence of cylindrical Wiener processes $\widetilde{W}%
^{N}=\sum_{k\in 
\mathbb{Z}
_{0}^{2}}e_{k}\widetilde{\beta }^{N,k}$ to $\widetilde{W}=\sum_{k\in 
\mathbb{Z}
_{0}^{2}}e_{k}\widetilde{\beta }^{k}$ on the suitable Hilbert space $%
L^{2}\left( \Pi ^{2}\right) $.

On the other hand, we have from $Step~I$ that 
\begin{eqnarray*}
\Psi \left( \widetilde{X}^{N}\right)&\rightharpoonup &\chi \text{ weakly
in }L^2(\Omega;L^{2}\left( 0,T;H^{1}\left( \Pi ^{2}\right) \right))\\
g\left( \widetilde{X}^{N}\right)  &\rightharpoonup &\rho \text{ weakly in }%
L^2(\Omega;L^{2}\left( 0,T;H^{1}\left( \Pi ^{2}\right) \right) ).
\end{eqnarray*}

Since $\Delta \Psi $ and $\Delta g$ are maximal monotone operators in $%
H^{-1}\left( \Pi ^{2}\right) $ and $\Psi$ and $g$ are maximal monotone operators in $L^{2}$ we get by combining the strong and the weak
convergences in a classical way that $\Psi \left( \widetilde{X}\right) =\chi $ and $g\left( 
\widetilde{X}\right) =\rho $.

We can now pass to the limit in 
\begin{eqnarray*}
\left( \widetilde{X}^{N}\left( t\right) ,e_{j}\right) _{2} &=&\left(
x,e_{j}\right) _{2}+\int_{0}^{t}\left( F\left( s\right) ,e_{j}\right) _{2}ds
\\
&&+\int_{0}^{t}\left( \Psi \left( \widetilde{X}^{N}\left( s\right) \right)
,\Delta e_{j}\right) _{2}ds+\int_{0}^{t}\left( g\left( \widetilde{X}%
^{N}\left( s\right) \right) ,\Delta e_{j}\right) _{2}ds \\
&&+\sum_{k\in 
\mathbb{Z}
_{0}^{2}}\int_{0}^{t}\alpha _{k}^{N}\left( \sigma _{k}\Gamma \left( 
\widetilde{X}^{N}\left( s\right) \right) ,\nabla e_{j}\right) d\widetilde{%
\beta }_{k}\left( s\right) 
\end{eqnarray*}%
as follows.
\begin{enumerate}
    \item First, let us note that the left-hand term converges $\mathbb{P}$-a.s. to  $\left( \widetilde{X}\left( t\right) ,e_{j}\right) _{2}$.
    \item For the stochastic term, we use the Burkholder-Davis-Gundi inequality and we
have%
\begin{eqnarray*}
&&\mathbb{E}\left[ \underset{t\in \left[ 0,T\right] }{\sup }\left\vert
\sum_{k\in 
\mathbb{Z}
_{0}^{2}}\int_{0}^{t}\alpha _{k}^{N}\left( \Gamma \left( \widetilde{X}%
^{N}\left( s\right) \right) ,\sigma _{k}\nabla e_{j}\right) _{2}d\widetilde{%
\beta }_{k}\left( s\right) \right\vert ^{2}\right]  \\
&\leq &\mathbb{E}\left[ \sum_{k\in 
\mathbb{Z}
_{0}^{2}}\int_{0}^{T}\left( \alpha _{k}^{N}\right) ^{2}\left( \Gamma \left( 
\widetilde{X}^{N}\left( s\right) \right) ,\sigma _{k}\nabla e_{j}\right)
_{2}^{2}ds\right]  \\
&\leq &\left\Vert \alpha ^{N}\right\Vert _{l^{\infty }}^{2}\mathbb{E}\left[
\int_{0}^{T}\left\Vert \Gamma \left( \widetilde{X}^{N}\left( s\right)
\right) \nabla e_{j}\right\Vert _{2}^{2}ds\right]  \\
&\leq &C\left\Vert \alpha ^{N}\right\Vert _{l^{\infty }}^{2}\lambda _{j}^{2}%
\mathbb{E}\left[ \int_{0}^{T}\left\Vert \widetilde{X}^{N}\left( s\right)
\right\Vert _{2}^{2}ds\right] \rightarrow 0.
\end{eqnarray*}%
for $N\rightarrow \infty $. 
\item We deduce that both the left-hand member and the stochastic integral converge to $\left( \widetilde{X}\left( t\right) ,e_{j}\right) _{2}$ and $0$ respectively in $L^2(\Omega;L^2(0,T;\mathbb{R}))$ in the strong, hence the weak topology. Passing to the limit in this space and in the weak sense, we get \begin{eqnarray*}
\left( \widetilde{X}\left( t\right) ,e_{j}\right) _{2} &=&\left(
x,e_{j}\right) _{2}+\int_{0}^{t}\left( F\left( s\right) ,e_{j}\right) _{2}ds
\\
&&+\int_{0}^{t}\left( \Psi \left( \widetilde{X}\left( s\right) \right)
,\Delta e_{j}\right) _{2}ds+\int_{0}^{t}\left( g\left( \widetilde{X}\left(
s\right) \right) ,\Delta e_{j}\right) _{2}ds \\
&&+\underset{N\rightarrow \infty }{\lim }\sum_{k\in 
\mathbb{Z}
_{0}^{2}}\int_{0}^{t}\alpha _{k}^{N}\left( \Gamma \left( \widetilde{X}%
^{N}\left( s\right) \right) ,\sigma _{k}\nabla e_{j}\right) _{2}d\widetilde{%
\beta }_{k}\left( s\right).
\end{eqnarray*}
\end{enumerate}

This concludes our proof.
\end{proof}
\begin{remark}
    The aforementioned argument only relies on the weak convergence in $L^2(\Omega;L^2(0,T;H^1(\Pi^2)))$ of the sequence $X^N$. Alternatively, owing to the estimates \eqref{estim3} that are uniform in $\omega$, one can deduce the weak convergence of $X^N$ (or, more precisely, of a subsequence) in a $\mathbb{P}$-a.s. sense. This can be achieved by applying, for instance, \cite[Theorem 3.1]{Yannelis}.
\end{remark}
\section*{Acknowledgments}
I.C. was supported in parts by the DEFHY3GEO project funded by Région Normandie and European Union with ERDF fund (convention 21E05300).\\
The research of F.F. is funded by the European Union, ERC NoisyFluid, No. 101053472.\\
D.G. acknowledges financial support from National Sciences and Engineering Research Council (NSERC), Canada, 
Grant/Award Number: RGPIN-2025-03963.

\end{document}